\documentclass[12pt, a4paper]{amsart}
\usepackage{hyperref, fullpage, amsmath, amsfonts, amsthm, amssymb, stmaryrd}

\theoremstyle{plain}
\newtheorem{Theorem}{Theorem}[section]
\newtheorem{Lemma}[Theorem]{Lemma}
\newtheorem{Proposition}[Theorem]{Proposition}
\newtheorem{Corollary}[Theorem]{Corollary}
\newtheorem{Observation}[Theorem]{Observation}

\theoremstyle{definition}
\newtheorem{Definition}[Theorem]{Definition}

\theoremstyle{remark}
\newtheorem{Remark}[Theorem]{Remark}
\newtheorem{Construction}[Theorem]{Construction}
\newtheorem{Example}[Theorem]{Example}

\author{Arno Fehm}
\address{Fachbereich Mathematik und Statistik, University of 
Konstanz, 78457 Konstanz, Germany}
\email{arno.fehm@uni-konstanz.de}

\author{Franziska Jahnke}
\address{Institut f\"ur Mathematische Logik und Grundlagenforschung,
University of M\"unster, Einsteinstr.\;62, 48149 M\"unster, Germany}
\email{franziska.jahnke@wwu.de}

\begin{document}

\title{On the quantifier complexity of 
definable canonical henselian valuations}

\begin{abstract}
We discuss definability in the language of rings without parameters of the unique canonical henselian valuation of a field.
We show that in most cases where the canonical henselian valuation is definable, it is already definable
by a universal-existential or an existential-universal formula.
\end{abstract}

\maketitle

\section{Introduction} 
A number of new results about
definability and definitions of henselian valuations without parameters in the 
language of rings have been proven recently. Here, a valuation $v$ on a field
$K$ is called
\emph{$\emptyset$-definable} if its valuation ring 
$\mathcal{O}_v$ is a first-order 
parameter-free definable subset of the field $K$.
Whereas some of the new developments focus more on the existence of definable 
henselian valuations (\cite{H}, 
\cite{JK0}), others put a further emphasis on the
quantifier complexity of the formulae involved 
(\cite{CDLM}, \cite{AK}, \cite{F}). Inspired
by the latter, Prestel has proven characterizations 
when a valuation in an elementary
class of valued fields
is uniformly $\emptyset$-$\exists$-definable, 
$\emptyset$-$\forall$-definable, $\emptyset$-$\forall\exists$-definable
or 
$\emptyset$-$\exists\forall$-definable (see Theorem \ref{P} and \cite{P}). 
These criteria work via the compactness theorem
and hence only give the existence of, rather than explicit, formulae.
A natural question arising from Prestel's 
results is whether indeed every $\emptyset$-definable henselian valuation is 
already $\emptyset$-$\forall\exists$-definable
or
$\emptyset$-$\exists\forall$-definable.

Since a field can carry a vast amount of inequivalent henselian valuations -- some of which are definable, and some of which are not -- it seems hopeless to get a general classification of the quantifier complexity of arbitrary definable henselian valuations.
However, every field $K$ carries a unique {\em canonical} henselian valuation,
and the task of classifying those according to their quantifier complexity turns out to be much more sensible and feasible.
Unless $K$ is separably closed, 
this canonical henselian valuation 
is non-trivial whenever $K$ admits some non-trivial henselian valuation (in which case one also calls the field itself {\em henselian})
and in many cases is {\em the} most interesting henselian valuation on $K$.
The goal of this paper is to show that, at least in residue characteristic zero, apart from very exceptional situations, the canonical henselian valuation on a field $K$
is always $\emptyset$-$\forall\exists$-definable or $\emptyset$-$\exists\forall$-definable, as soon as it is $\emptyset$-definable at all.

We first treat the simplified setting of canonical {\em $p$-henselian} valuations (cf.~Section~\ref{sec:p} for definitions and details)
in which we get the best result one can hope for: Depending on whether its residue field is $p$-closed or not, the canonical $p$-henselian valuation
is either $\emptyset$-$\forall\exists$-definable or $\emptyset$-$\exists\forall$-definable whenever it is
$\emptyset$-definable at all (see Propositions \ref{p1} and \ref{p2}).

Although the definition of the canonical {\em henselian} valuation (which we recall in Section~\ref{sec:not}) suggests a case distinction between separably closed and non-separably closed residue field, it turns out that here the dividing line between $\exists\forall$ and $\forall\exists$ runs somewhere else:

\begin{Theorem} \label{main}
Let $K$ be a field with canonical henselian valuation $v_K$ whose residue field $F=Kv_K$ has characteristic zero.
Assume that $v_K$ is $\emptyset$-definable.
\begin{enumerate}
\item If $F$ \underline{is not} elementarily equivalent to a henselian field, then $v_K$ is $\emptyset$-$\exists\forall$-definable.
\item If $F$ \underline{is} elementarily equivalent to a henselian field,
 then $v_K$ is $\emptyset$-$\forall\exists$-definable if the absolute Galois group $G_F$ of $F$ is a small profinite group.
\end{enumerate}
\end{Theorem}
Recall that $G_F$ is small iff $F$ has only finitely many Galois extensions of degree $n$ for every $n \in \mathbb{N}$.
Thus, case (2) includes in particular the important cases, let us call them (2a) and (2b), where $F$ is algebraically closed resp.\ real closed.

The proof of (1) is straightforward and does not even need the assumption that $v_K$ is $\emptyset$-definable (Proposition \ref{Prop:H1}).
Also in case (2a), a direct proof gives a stronger result than stated here (Corollary \ref{MC}).
The general case (2) is more difficult to handle.
In fact, if the residue field of the canonical henselian valuation is not separably closed, then it is never henselian.
Except for separably or real closed fields, very few examples of 
fields $F$ that are not henselian but elementarily equivalent to a henselian field are known, and we do not know whether in this case the canonical henselian valuation on $K$ is always
$\emptyset$-$\forall\exists$-definable or $\emptyset$-$\exists\forall$-definable in general.
The case we can handle, namely when the absolute Galois group is small,
is proven using, among other things, Koenigsmann's Galois characterization of 
tamely branching valuations (Theorem \ref{small}).
In the last section, we construct an example to demonstrate that such fields with small absolute Galois groups do exist (Proposition \ref{Prop:small}).
We also construct several examples along the way to show that, in general, our results can not be improved in terms of quantifier complexity (see 
Examples \ref{Ex:H2}, \ref{Ex:H1p}, \ref{Ex:H22}, \ref{Ex:H1}, and \ref{Ex:t}).

\section{Notation and some facts} \label{sec:not}
Throughout this paper, we use the following notation: 
For a valued field $(K,v)$,
we denote the valuation ring by ${\mathcal O}_v$, the maximal ideal of $\mathcal{O}_v$ by ${\mathfrak m}_v$, the residue field by $Kv$
and the value group by $vK$.
For an element $a \in {\mathcal O}_v$, we write 
$\overline{a}$ to refer to its image in $Kv$.
For valuations $v$ and $w$ on $K$ we write $v\subseteq w$ to indicate that $v$ is finer than $w$, i.e.\ $\mathcal{O}_v\subseteq\mathcal{O}_w$.
We denote by $K^{\rm sep}$ a fixed separable closure of $K$ and by $G_K={\rm Gal}(K^{\rm sep}|K)$ the absolute Galois group of $K$.

Several of our examples involve power series fields. For a field $F$ and an
ordered abelian group $\Gamma$, we write
$$
 F((\Gamma))=F((t^\Gamma))=\Big\{ \sum\limits_{\gamma \in \Gamma} a_\gamma t^\gamma \mid a_\gamma\in F, \{\gamma \in \Gamma \mid a_\gamma \neq 0\} 
\textrm{ is well-ordered}\Big\}
$$
for the field of generalized power series over $F$ with exponents in $\Gamma$.
The {\em power series valuation} 
$$ 
 v\Big(\sum\limits_{\gamma \in \Gamma} a_\gamma t^\gamma\Big) := \mathrm{min}\{\gamma \in \Gamma \mid a_\gamma \neq 0\}
$$
is a henselian valuation on $F((\Gamma))$ with residue field $F$ and value group $\Gamma$. 
We write $F(t^\Gamma)$ for the subfield of $F((t^\Gamma))$ generated over $F$ by the monomials $t^\gamma$ for $\gamma\in\Gamma$.
See \cite[\S4.2]{Efrat} for more details of this construction. 
If $\Gamma_1$ and $\Gamma_2$ are ordered abelian groups we denote by $\Gamma_1\oplus\Gamma_2$ their
inverse lexicographic product.
There is then a natural isomorphism $F((\Gamma_1\oplus\Gamma_2))\cong F((\Gamma_1))((\Gamma_2))$.

All our definitions will be obtained from the following theorem of Prestel \cite[Characterization Theorem]{P}:
\begin{Theorem}[Prestel] \label{P}
Let $\Sigma$ be a first order axiom system in the ring language
$\mathcal{L}_{\rm ring}$ together with a unary predicate $\mathcal{O}$. Then there exists an
$\mathcal{L}_{\rm ring}$-formula $\phi(x)$, defining uniformly
in every model $(K, \mathcal{O})$ of $\Sigma$ the set $\mathcal{O}$, 
of quantifier type
\begin{align*}
\exists & \textrm{ iff }(K_1 \subseteq K_2 \Rightarrow \mathcal{O}_1
\subseteq \mathcal{O}_2) \\
\forall & \textrm{ iff }(K_1 \subseteq K_2 \Rightarrow \mathcal{O}_2 \cap K_1
\subseteq \mathcal{O}_1) \\
\exists\forall & \textrm{ iff }(K_1 \prec_\exists K_2 \Rightarrow \mathcal{O}_1
\subseteq \mathcal{O}_2) \\
\forall\exists & \textrm{ iff }(K_1 \prec_\exists K_2 \Rightarrow \mathcal{O}_2
\cap K_1 \subseteq \mathcal{O}_1) 
\end{align*}
for all models $(K_1, \mathcal{O}_1)$, $(K_2, \mathcal{O}_2)$ of $\Sigma$.
Here $K_1 \prec_\exists K_2$ means that $K_1$ is existentially closed in
$K_2$, i.e.\;every existential 
$\mathcal{L}_{\rm ring}$-formula $\rho(x_1, \dotsc,x_m)$
with parameters from $K_1$ that holds in $K_2$ also holds in $K_1$.
\end{Theorem}

We use the above theorem in later sections 
to show that in order to define the canonical
henselian valuation without parameters, only formulae of a low quantifier 
complexity are needed. 
We call a field $K$ \emph{henselian} if it admits some non-trivial henselian
valuation. 
There is always
a \emph{canonical henselian valuation} on $K$. We now
recall the definition and its defining properties, details
can be found in section 4.4 of \cite{EP}.

If a field admits two independent non-trivial henselian valuations, 
then it is separably closed. This implies that 
the henselian valuations on a field
form a tree: Divide the class of henselian valuations on $K$ into
two subclasses, namely
$$H_1(K) = \{v \textrm{ henselian on } K \;|\; Kv \neq Kv^{\rm sep} \}$$
and
$$H_2(K) = \{ v \textrm{ henselian on } K \;|\; Kv = Kv^{\rm sep} \}.$$
Then, any valuation $v_2 \in H_2(K)$ 
is strictly \emph{finer} than any $v_1 \in H_1(K)$, 
i.e.~${\mathcal O}_{v_2} \subsetneq {\mathcal O}_{v_1}$,
and any two valuations in $H_1(K)$ are comparable.
Furthermore, if $H_2(K)$ is non-empty, then there exists a unique coarsest
$v_K \in H_2(K)$; otherwise there exists a unique finest $v_K \in H_1(K)$.
In either case, $v_K$ is called the \emph{canonical henselian valuation}.
Note that if $K$ is not separably closed and 
admits a non-trivial henselian valuation, then $v_K$
is non-trivial.

The definition of the canonical henselian valuation motivates the following
\begin{Lemma} \label{Lemhen}
Let $K \subset L$ be an extension of fields such that $K$ is relatively
algebraically closed in $L$. Let $w$ be a henselian valuation on $L$. Then
all of the following hold:
\begin{enumerate}
\item The restriction $v$ of $w$ to $K$ is also henselian.
\item $Kv$ is separably closed in $Lw$.
\item If $Lw$ is separably closed, then $Kv$ is also separably closed.
\end{enumerate}
\end{Lemma} 

\begin{proof}
\begin{enumerate}
\item 
See \cite[4.1.5]{EP}.
\item Let $f\in\mathcal{O}_v[X]$ monic such that $\bar{f}\in Kv[X]$ is separable and irreducible and has a zero in $Lw$.
Since $(L,w)$ is henselian, $f$ has a zero $a$ in $L$, which, since $K$ is algebraically closed in $L$, lies in $K$, hence in $\mathcal{O}_v$.
Thus, $\bar{a}\in Kv$ satisfies $\bar{f}(\bar{a})=0$, so ${\rm deg}(f)=1$.
\item This follows immediately from (2).
\end{enumerate}
\end{proof}

In general, the canonical henselian valuation need not be
$\emptyset$-definable. Whenever it is $\emptyset$-definable, this might
be for the `right' or for the `wrong' reason, see also the discussion in 
\cite[p.~3]{JK}. 
This motivates the next
\begin{Definition}
We say that $v_K$ is $\emptyset$-definable
{\em as such} if there is a parameter-free $\mathcal{L}_{\rm ring}$-formula
$\phi(x)$ such that for all fields $L$ with $L\equiv K$,
we have $\phi(L)=\mathcal{O}_{v_L}$.
\end{Definition}

Using Theorem \ref{P}, we can now draw some first conclusions about 
the quantifier complexity of definitions of the canonical henselian
valuation:
\begin{Observation}\label{Obshen}
Assume that $v_K$ is $\emptyset$-definable as such. Then,
\begin{enumerate}
\item if $v_K \in H_1(K)$, then $v_K$ is $\emptyset$-$\exists\forall$-definable,
\item if $v_K \in H_2(K)$, then $v_K$ is $\emptyset$-$\forall\exists$-definable.
\end{enumerate}
\end{Observation}
\begin{proof}
Note that if 
$v_K$ is $\emptyset$-definable as such, then we have 
$v_K \in H_2(K)$ iff $v_L \in H_2(L)$ for any $L \equiv K$.
Take $L,M \equiv K$ such that $L \prec_\exists M$, so in particular $L$ is 
relatively algebraically closed in $M$. By Lemma \ref{Lemhen},
the restriction $w$ of $v_M$ to $L$ is again henselian. 
\begin{enumerate} 
\item If
$v_L \in H_1(L)$, then $H_2(L)=\emptyset$, hence $w$
must be a coarsening of $v_L$. Hence $v_K$ is
$\emptyset$-$\exists\forall$-definable by Theorem \ref{P}.
\item In case $v_M \in H_2(M)$, Lemma \ref{Lemhen} implies that
$Lw=Lw^{\rm sep}$. Thus $w$ is a refinement of $v_L$ and so $v_K$ is
$\emptyset$-$\forall\exists$-definable by Theorem \ref{P}.
\end{enumerate}
\end{proof}

As we will see later on, in both cases the definitions are optimal with regard
to quantifiers:
In Example \ref{Ex:H1}, we construct a field $K$ with $v_K \in H_1(K)$ such that
$v_K$ is $\emptyset$-definable as such but not 
$\emptyset$-$\forall\exists$-definable.
Similarly, we discuss a field $K$ with 
$v_K \in H_2(K)$ and such that $v_K$ is $\emptyset$-definable as such but not
$\emptyset$-$\exists\forall$-definable in Example \ref{Ex:H22}. In particular, 
$v_K$ is in both cases in general 
neither $\emptyset$-$\exists$- nor $\emptyset$-$\forall$-definable.

\section{The canonical $p$-henselian valuation} \label{sec:p}

In this section, we discuss the canonical $p$-henselian valuation and prove analogues of the observation in the previous section.

Let $p$ be a prime and $K$ a field. 
If ${\rm char}(K)\neq p$, we denote by $\zeta_p$ a primitive $p$-th root of unity in $K^{\rm sep}$.
We define $K(p)$ to be the compositum
of all Galois extensions of $K$ of $p$-power degree inside $K^{\rm sep}$.
\begin{Definition} A valuation $v$ on $K$ is called
\emph{$p$-henselian} if $v$ extends uniquely to $K(p)$.
We call $K$ \emph{$p$-henselian} if $K$ 
admits a non-trivial
$p$-henselian valuation. 
\end{Definition}

As with henselian valuations, there is an equivalent definition involving
the lifting of zeroes from the residue field:
\begin{Proposition}[{\cite[Proposition 1.2]{Koe}}] \label{phenseq}
For a valued field $(K,v)$, the following are equivalent:
\begin{enumerate}
\item $v$ is $p$-henselian,
\item for every polynomial $f \in {\mathcal O}_v$ which splits in $K(p)$ and
every $a \in {\mathcal O}_v$ with $\bar{f}(\overline{a}) = 0$ and
$\bar{f'}(\overline{a}) \neq 0$, there exists $\alpha
\in {\mathcal O}_v$ with $f(\alpha)=0$ and $\overline{\alpha}=\overline{a}$.
\end{enumerate}
\end{Proposition}

The following facts can be found in \cite{Koe}.
If $K$ admits two independent non-trivial $p$-henselian valuations, 
then $K=K(p)$. We can once more divide the 
class of $p$-henselian valuations on $K$ into
two subclasses, namely
$$
 H_1^p(K) = \{v \textrm{ $p$-henselian on } K\; |\; Kv \neq Kv(p) \}
$$
and
$$
 H_2^p(K) = \{ v \textrm{ $p$-henselian on } K\; |\; Kv = Kv(p) \}.
$$
Then, any valuation $v_2 \in H_2^p(K)$ 
is strictly \emph{finer} than any $v_1 \in H_1^p(K)$, 
and any two valuations in $H_1^p(K)$ are comparable.
Furthermore, if $H_2^p(K)$ is non-empty, then there exists a unique coarsest
$v_K^p \in H_2^p(K)$; otherwise there exists a unique finest 
$v_K^p \in H_1^p(K)$.
In either case, $v_K^p$ is called the \emph{canonical $p$-henselian valuation}.
Note that if $K\neq K(p)$ admits a non-trivial $p$-henselian valuation, then $v_K^p$
is also non-trivial.

We get the following variant of Lemma \ref{Lemhen}:
\begin{Lemma} \label{Lemphen}
Let $K \subset L$ be an extension of fields such that $K$ is $p$-closed in $L$, i.e.\ $K(p)\cap L=K$.
Let $w$ be a $p$-henselian valuation on $L$. Then
the following holds:
\begin{enumerate}
\item The restriction $v$ of $w$ to $K$ is also $p$-henselian.
\item $Kv$ is $p$-closed in $Lw$, i.e.\ $Kv(p)\cap Lw=Kv$.
\item If $Lw=Lw(p)$, then $Kv=Kv(p)$.
\end{enumerate}
\end{Lemma} 

\begin{proof} 
\begin{enumerate}
\item The assumption $K(p)\cap L=K$ implies that $L$ and $K(p)$ are linearly disjoint over $K$,
so every extension of $v$ to $K(p)$ is the restriction of the unique extension of $w$ to $K(p)L\subseteq L(p)$.
\item Let $g\in Kv[X]$ be of degree $p$ that splits in $Kv(p)$ and has zero in $Lw$.
By \cite[4.2.6]{EP}, there is $f\in \mathcal{O}_v[X]$ monic of degree $p$ with $\bar{f}=g$ such that $f$ splits in $K(p)$.
Since $(L,w)$ is $p$-henselian, $f$ has a zero $a$ in $L$, which by $L\cap K(p)=K$ lies in $\mathcal{O}_v$,
so $g(\bar{a})=0$ and $g$ splits already in $Kv$.
Since every Galois extension of $p$-power degree contains a Galois extension of degree $p$,
this proves the claim.
\item This follows immediately from (2), as $Kv(p)\subseteq Lw(p)$.
\end{enumerate}
\end{proof}

\begin{Definition}
We say that $v_K^p$ is $\emptyset$-definable
{\em as such} if there is a parameter-free $\mathcal{L}_\textrm{ring}$-formula
$\phi_p(x)$ such that for all fields $L$ with $L \equiv K$,
we have $\phi_p(L)=\mathcal{O}_{v_L^p}$.
\end{Definition}

Unlike the canonical henselian valuation, the canonical $p$-henselian
valuation is usuall
y $\emptyset$-definable as such. 
Recall that a field $K$ is {\em Euclidean} if $[K(2):K]=2$.

\begin{Theorem}[{\cite[Main Theorem]{JK}}] \label{JKmain}
Fix a prime $p$. There exists a parameter-free 
$\mathcal{L}_\textrm{ring}$-formula $\phi_p(x)$
such that for any field $K$ with either $\mathrm{char}(K)=p$ or $\zeta_p \in K$
the following are equivalent:
\begin{enumerate}
\item $\phi_p$ defines ${v_K^p}$ as such.
\item $v_K^p$ is $\emptyset$-definable as such.
\item $p\neq 2$ or $Kv_K^p$ is not Euclidean.
\end{enumerate}
\end{Theorem}

We can now prove the $p$-henselian analogues of Observation \ref{Obshen}.
\begin{Proposition} Let $p$ be a prime. \label{p1}
Consider the elementary class of valued fields
$$
\mathcal K := \{ (K,v) \mid 
v=v_K^p \in H_1^p(K), \zeta_p \in K 
\textrm{ if char}(K)\neq p, Kv\mbox{ not Euclidean if }p=2\}$$
Then $v_K^p$ is uniformly $\emptyset$-$\exists\forall$-definable for all
$K$ with $(K,v_K^p) \in \mathcal K$.
\end{Proposition}
\begin{proof}
Note that $\mathcal{K}$ is elementary by Theorem \ref{JKmain}.
Take $(L,v_L^p),(M,v_M^p) \in \mathcal{K}$ such that $L \prec_\exists M$. 
By Lemma \ref{Lemphen}, the restriction 
$w$ of $v_M^p$ to $L$ is $p$-henselian.
As $v_L^p \in H_1^p(L)$, $v_L^p$ is the finest $p$-henselian
valuation on $L$, and thus we get 
$\mathcal{O}_{v_L^p} \subseteq \mathcal{O}_w \subseteq
\mathcal{O}_{v_M^p}$ and hence uniform 
$\emptyset$-$\exists\forall$-definability by Theorem \ref{P}.
\end{proof}

\begin{Proposition}\label{p2} Let $p$ be a prime.
Consider the elementary class of valued fields
$$
 \mathcal{K} := \{(K,v) \mid v=v_K^p \in H_2^p(K)\mbox{ and } \zeta_p\in K\mbox{ if }{\rm char}(K)\neq p\}.
$$
Then $v_K^p$ is uniformly $\emptyset$-$\forall\exists$-definable 
for all $K$ with $(K,v_K^p)\in \mathcal K$.
\end{Proposition}

\begin{proof}
Note that $\mathcal{K}$ is elementary by Theorem \ref{JKmain}.
Take $(L,v_L^p),(M,v_M^p) \in \mathcal{K}$ such that $L \prec_\exists M$. 
Using Lemma \ref{Lemphen} again, the restriction $w$ of $v_M^p$ to $L$
is $p$-henselian and we have $Lw = Lw(p)$, so $w\in H_2^p(L)$
and therefore $\mathcal{O}_w\subseteq\mathcal{O}_{v_L}$.
Thus, we get uniform
$\emptyset$-$\forall\exists$-definability by Theorem \ref{P}.
\end{proof}

Theorem \ref{JKmain} includes an exception in case $p=2$ and 
$Kv_K^2$ is Euclidean. However, in this case some coarsening of $v_K^2$
is nonetheless
$\emptyset$-definable:

\begin{Proposition}[{\cite[Observation 2.3]{JK}}] \label{2def}
Let $K \neq K(2)$, and assume that $Kv_K^2$ is Euclidean.
Then the coarsest $2$-henselian valuation $v_K^{2*}$ on $K$ 
which has Euclidean
residue field is $\emptyset$-definable.
\end{Proposition}

Again, this definition can be found to be of type $\forall\exists$:
\begin{Proposition} \label{euc}
Consider the elementary class of valued fields
$$
 \mathcal K := \{ (K,v) \mid Kv_K^2 \textrm{ is Euclidean and }v=v_K^{2*}\}
$$
Then $v_K^{2*}$ is uniformly $\emptyset$-$\forall\exists$-definable for all
$K$ with $(K,v_K^{2*}) \in \mathcal K$.
\end{Proposition}
\begin{proof}
The class of fields
$$\mathcal{K}_0:= \{ K \mid Kv_K^2 \textrm{ Euclidean} \}$$ 
is elementary by \cite[Observation 2.3(b)]{JK}.
Furthermore, the proof of \cite[Observation 2.3(a)]{JK} shows that
$v_K^{2*}$ is uniformly $\emptyset$-definable in any $K \in \mathcal{K}_0$.
Thus, $\mathcal{K}$ is an elementary class of valued fields.
 
The rest of the proof is similar to the one of Proposition \ref{p2}:
Take $(L,v_L^{2*}), (M,v_M^{2*}) \in \mathcal{K}$ such that 
$L \prec_\exists M$. 
Using Lemma \ref{Lemphen}, the restriction $w$ of $v_M^{2*}$ to $L$
is $2$-henselian and we have 
$Lw(2)\cap Mv_M^{2*}=Lw$.
This implies that $[Lw(2):Lw]\leq[Mv_M^{2*}(2):Mv_M^{2*}]=2$, as $Mv_M^{2*}$ is Euclidean.
Therefore, since $v_L^2\in H_1^2(L)$ implies that $Lw\neq Lw(2)$, we conclude that $[Lw(2):Lw]=2$, 
so $Lw$ is Euclidean.
In particular, $w$ is a refinement of $v_L^{2*}$. 
Thus, Theorem \ref{P} implies 
uniform $\emptyset$-$\forall\exists$-definability.  
\end{proof}

In general, $v_K^p$ need not be \emph{simultaneously} 
$\forall\exists$- and $\exists\forall$-definable without parameters:
\begin{Example} \label{Ex:H2}
Fix a prime $p$. 
We construct a field $K$ with $v_K^p \in H_2^p(K)$ such that $v_K^p$ is $\emptyset$-$\forall\exists$-definable as such but not $\emptyset$-$\exists\forall$-definable.

Consider the field $K_0=\mathbb{C}((\mathbb{Q}))$ and 
let $H=\mathbb{Z}\oplus\mathbb{Q}$ 
(recall that the direct sum is ordered {\em inverse} lexicographically).
In particular, $H$ is discrete and there is an embedding of ordered groups $\mathbb{Q}\rightarrow H$.
Let $D$ be the divisible hull of $H$. 
Note that the theory of divisible ordered abelian groups is
 model complete (see \cite[3.1.17]{Mar}). 
So, as $D$ contains $\mathbb{Q}$, we have $\mathbb{Q} \prec D$ in the language 
of ordered groups. This implies in particular 
$\mathbb{Q} \prec_\exists H$ (as ordered
abelian groups).
Take a copy $H_i$ of $H$ for each $i\geq0$ and let
$$
 \Gamma=H_1\oplus H_2\oplus\dots,
$$
again with inverse lexicographic order.
Now \cite[Corollary 1.7]{We} yields that
$$
 G_1:=\mathbb{Q}\oplus \Gamma
\prec_\exists H_0\oplus\Gamma =:G_2$$
as ordered abelian groups.
Consider the fields $K_1=\mathbb{C}((G_1))\cong K_0((\Gamma))$ and  $K_2=\mathbb{C}((G_2))$. 
For $i=1,2$, let $v_i$ denote the henselian valuation on $K_i$ with
value group $G_i$ and residue field $\mathbb{C}$,
and let $w$ denote the henselian valuation on $K_1$ with value group $\Gamma$ and residue field $K_0$.
Then the Ax-Kochen/Ersov-Theorem (see \cite[4.6.4]{P2}) implies
$(K_1,w) \equiv (K_2,v_2)$ since $K_0 \equiv \mathbb{C}$
and $\Gamma \cong G_2$.
Moreover, $(K_1,v_1) \prec_\exists (K_2,v_2)$ by a well-known variant of
the Ax-Kochen/Ersov-Theorem (see p.\,183 in \cite{KP}),
thus we get in particular $K_1 \prec_\exists K_2$ in the language of rings. 
However,
$v_{K_1}^p=w$ (since $\Gamma$ is discrete) and $v_{K_2}^p=v_2$.
Hence, the restriction of $v_{K_2}^p$ to $K_1$ is a 
proper refinement of $v_{K_1}^p$. Thus, the canonical 
$p$-henselian valuation on $K_1$ is not $\emptyset$-$\exists\forall$-definable 
by Theorem \ref{P}, although it is $\emptyset$-$\forall\exists$-definable as such by Proposition \ref{p2}. 
Note that in fact $v_K^p$ is henselian, so this also gives
an example of an $\emptyset$-$\forall\exists$-definable henselian
valuation which is not $\emptyset$-$\exists\forall$-definable
(cf.~Example \ref{Ex:H22}).
\end{Example}

Conversely, we give an example of a canonical $p$-henselian valuation which is
$\emptyset$-$\exists\forall$-definable 
but not $\emptyset$-$\forall\exists$-definable: 
\begin{Example} \label{Ex:H1p} Fix any prime $p$.
We construct a field $K$ with $v_K^p \in H_1^p(K)$ such that $v_K^p$ is
$\emptyset$-$\exists\forall$-definable as such but not $\emptyset$-$\forall\exists$-definable.

We first construct a field $k$ with $k \cong k((\mathbb{Q}))(X)$ containing
a primitive $p$th root of unity $\zeta_p$:  
For $i \geq 1$, 
let
$$
 k_{i+1} := \mathbb{C}((t_{i+1}^\mathbb{Q}))(X_{i+1})((t_{i}^\mathbb{Q}))(X_{i}) \dots((t_{1}^\mathbb{Q}))(X_{1}).
$$ 
and 
$$
 k := \bigcup\limits_{i\geq 1} k_i.
$$
Then $k \cong k((t_0^\mathbb{Q}))(X_0)$ by mapping 
$$
 X_i\mapsto X_{i-1} \textrm{ and } t_i \mapsto t_{i-1} \textrm{ for } i>0.
$$
Take $L_1 := k((u^\mathbb{Q}))((v^\mathbb{Q}))$ and 
$L_2:=k((u^\mathbb{Q}))(X)((v^\mathbb{Q}))$. 
Denote by $v_1$ the power series valuation on $L_1$ with value group 
$\mathbb{Q} \oplus \mathbb{Q}$ and residue field $k$,
and by $v_2$ the power series valuation on $L_2$ with value group $\mathbb{Q}$ and residue field $k((u^\mathbb{Q}))(X)$.
Then, by Ax-Kochen/Ersov 
(\cite[4.6.4]{P2}), we have
$$
 (L_1,v_1) \equiv (L_2,v_2)
$$ since $k \equiv k((u^\mathbb{Q}))(X)$ holds by construction and
since $\mathbb{Q}\oplus\mathbb{Q}$ is divisible and thus elementarily equivalent to $\mathbb{Q}$.
Furthermore, we have
$$L_1 \prec_\exists L_2$$
by a well-known Ax-Kochen/Ersov variant (see p.\,183 in \cite{KP}) 
since we have $\mathbb{Q} \prec_\exists \mathbb{Q}$ and
$$
 k((u^\mathbb{Q})) \prec_\exists k((u^\mathbb{Q}))(X),
$$
as every purely transcendental extension of a field can be embedded into any sufficiently large elementary extension.
Since $k$ is by construction hilbertian (see \cite[13.2.1]{FJ}), 
it is not $p$-henselian, not 
Euclidean and admits Galois extensions of degree $p$ 
(see \cite[Lemma 3.2]{JK0}).
Thus, $v_{L_1}^p=v_1$.
Furthermore, $v_{L_1}^p$ is $\emptyset$-definable 
as such by an $\exists\forall$-formula by Proposition
\ref{p1}.
On the other hand, 
$v_{L_2}^p=v_2$.
Thus, the restriction of $v_{L_2}^p$ to 
$L_1$ is a proper coarsening
of $v_{L_1}^p$ and so $v_{L_1}^p$ is not $\emptyset$-definable 
by an $\forall\exists$-formula
by Theorem \ref{P}. In fact, $v_{L_1}^p$ coincides with the canonical henselian
valuation, so this also gives rise to an example of a 
canonical henselian valuation which is not
$\emptyset$-$\forall\exists$-definable (cf.~Example \ref{Ex:H1}).
\end{Example}

\section{The case $v_K \in H_2(K)$}

Consider a field $K$ with
with canonical henselian valuation 
$v_K \in H_2(K)$ and $\mathrm{char}(Kv_K)=0$.
We now want to show that if $v_K$  
is $\emptyset$-definable on such a field, then it is already 
$\emptyset$-$\forall\exists$-definable.

\begin{Lemma} \label{C}
Let $\Gamma$ be any ordered abelian group. 
Consider the field $L=\mathbb{C}((\Gamma))$ and let $v$ denote the 
power series valuation on $L$.
Then no proper refinement of $v$ is $\mathbb{C}$-definable.
\end{Lemma}

\begin{proof} 
Let $w$ be a proper refinement of $v$ and suppose that $\mathcal{O}_w=\phi(L)$ for some
formula $\phi$ with parameters from $\mathbb{C}$.
Let $K_0\subseteq\mathbb{C}$ be an algebraically closed field of finite transcendence degree that contains those parameters.
As $w$ refines $v$, $w$ induces a non-trivial valuation $\overline{w}$ 
on the residue field $Lv=\mathbb{C}$, 
and since the residue map $\mathcal{O}_v\rightarrow\mathbb{C}$ is the identity on $\mathbb{C}$,
the restriction of $w$ to $\mathbb{C}$ equals $\overline{w}$. 
Thus, since $K_0$ is a proper subfield of $\mathbb{C}$,
there is some $a\in\mathbb{C}\setminus K_0$ with $w(a)>0$.
As ${\rm Aut}(\mathbb{C}|K_0)$ acts transitively on $\mathbb{C}\setminus K_0$,
there is some $\sigma\in{\rm Aut}(\mathbb{C}|K_0)$ with $\sigma(a)=a^{-1}$.
We can extend $\sigma$ to an automorphism $\sigma'\in{\rm Aut}(L|K_0)$ by
setting
$$
 \sigma'\left(\sum_{\gamma\in\Gamma} a_\gamma t^\gamma\right)=\sum_{\gamma\in\Gamma}\sigma(a_\gamma)t^\gamma.
$$
Since $\sigma'$ fixes the parameters of $\phi$, $\sigma'(\mathcal{O}_w)=\mathcal{O}_w$,
contradicting $\sigma'(a)=a^{-1}\notin\mathcal{O}_w$.
\end{proof}

\begin{Proposition} \label{Prop:H2}
Let $K$ be a field with $v_K \in H_2(K)$ and $\mathrm{char}(Kv_K)=0$.
Then no proper refinement of $v_K$ is $\emptyset$-definable.
\end{Proposition}

\begin{proof}
Let $w$ be a proper refinement of $v_K$ and suppose that $\mathcal{O}_w=\phi(K)$ for some formula $\phi$.
Since the theory of algebraically closed fields of characteristic $0$ is complete,
Ax-Kochen/Ersov (\cite[4.6.4]{P2}) implies that $(K,v_K)\equiv(L,v)$, where 
$L=\mathbb{C}((v_KK))$ and $v$ is the power series valuation on $L$.
Since this is an elementary equivalence of valued fields, and $\phi(K)\subsetneqq\mathcal{O}_{v_K}$,
also $\phi(L)\subsetneqq\mathcal{O}_v$, so $\phi$ defines a proper refinement of $v$,
which is impossible by Lemma \ref{C}.
\end{proof}

\begin{Corollary} \label{MC}
Let $K$ be a field with $v_K \in H_2(K)$ and $\mathrm{char}(Kv_K)=0$.
If $v_K$ is $\emptyset$-definable, then $v_K$ is $\emptyset$-definable
as such by an $\forall\exists$-formula.
\end{Corollary}
\begin{proof}
Let $\phi$ be a formula that defines $v_K$.
If $L\equiv K$, then $\phi(L)$ is a henselian valuation ring $\mathcal{O}_w$ with $w\in H_2(L)$ and ${\rm char}(Lw)=0$.
In particular, $v_L\in H_2(L)$ and ${\rm char}(Lv_L)=0$, 
so Proposition \ref{Prop:H2} implies that $w=v_L$,
hence $\phi$ defines $v_K$ as such.
The claim now follows from Observation \ref{Obshen}.
\end{proof}

\begin{Remark} Let $K$ be a field with $v_K \in H_2(K)$, 
$\mathrm{char}(Kv_K)=0$ and $K\neq K^{\rm sep}$.
Then for some prime $p$, $v_K^p$ is a non-trivial coarsening of $v_K$ 
(cf.~\cite[Theorem 3.10]{JK0}). Thus, since $\zeta_p\in K$, 
Proposition \ref{p2} shows that $K$ always admits
\emph{some} $\emptyset$-$\forall\exists$-definable henselian valuation.
\end{Remark}

\begin{Example} \label{Ex:not}
We construct a field $K$ with $v_K\in H_2(K)$ which is {\em not} $\emptyset$-definable:
Let $K=\mathbb{C}((\Gamma))$, where
$$
 \Gamma=\bigoplus_{p}\mathbb{Z}_{(p)}=\ldots\oplus\mathbb{Z}_{(5)}\oplus\mathbb{Z}_{(3)}\oplus\mathbb{Z}_{(2)}
$$
is ordered inverse lexicographically.
Here, $p$ runs over all prime numbers, and $\mathbb{Z}_{(p)}$ is the localization of $\mathbb{Z}$ at $p$.
For every prime $l$, the canonical 
$l$-henselian valuation on $K$ is the power series valuation 
on $K$ with value group $\bigoplus_{p \leq l}\mathbb{Z}_{(p)}$ 
and residue field 
$\mathbb{C}((\bigoplus_{p>l}\mathbb{Z}_{(p)}))$.
However, the canonical henselian valuation on $K$ is the power series valuation
on $K$ with residue field $\mathbb{C}$ and value group 
$\Gamma$. In particular, we have
$\mathrm{char}(Kv_K)=0$, $v_K \in H_2(K)$ and 
$v_K \subsetneq v_K^p$ for all primes $p$. 

We now use Proposition \ref{Prop:H2}  
to see that the canonical henselian valuation is not
$\emptyset$-definable on $K$: 
Note that $\Gamma$ has a nontrivial $p$-divisible subgroup for every prime $p$, thus 
$\Gamma\equiv\mathbb{Q}\oplus\Gamma$, see Lemma \ref{Gamma} below.
Now consider
$L := \mathbb{C}((\mathbb{Q}\oplus\Gamma))$ with the power series valuation $w$.
By the Ax-Kochen/Ershov Theorem (\cite[4.6.4]{P2}), $(K,v_K)\equiv(L,w)$.
If $v_K$ were $\emptyset$-definable, the same formula would define $w$ on $L$.
However, the canonical henselian valuation on $L$ has value group $\Gamma$ and residue field $\mathbb{C}((\mathbb{Q}))$,
so $w$ is a proper refinement of $v_L$, contradicting Proposition \ref{Prop:H2}.

Note that if $p < q$, we have $\mathcal{O}_{v_K^q} \subsetneq
\mathcal{O}_{v_K^p}$. Thus, there are countably many different henselian
valuations $\emptyset$-definable on $K$. Since $v_K$ is not 
$\emptyset$-definable, all $\emptyset$-definable henselian valuations on $K$
 are in $H_1(K)$ by Proposition \ref{Prop:H2}.
\end{Example}

\begin{Example} \label{Ex:H22} Recall that Example \ref{Ex:H2} discusses
a field $K_2:=\mathbb{C}((G_2))$ where $G_2$ is some ordered abelian group
with certain properties. We show there 
that the power series valuation $v_2$ on $K_2$
coincides with the canonical $p$-henselian valuation and is 
$\emptyset$-definable but
not 
$\emptyset$-$\exists\forall$-definable. However, we also have $v_2=v_{K_2}$,
so by Corollary \ref{MC}, the canonical henselian valuation
$v_{K_2}$ is $\emptyset$-definable as such 
but again not $\emptyset$-$\exists\forall$-definable.
\end{Example}

\section{The case $v_K \in H_1(K)$}

Let $K$ be a henselian field with $v_K \in H_1(K)$ and $\mathrm{char}(Kv_K)=0$.
Unlike in the case $v_K \in H_2(K)$,
it is not true that $v_K$ is already $\emptyset$-definable as such 
whenever it is
$\emptyset$-definable. In order to explain this, we need the following
\begin{Definition}
A field $K$ is called \emph{$t$-henselian} if there is 
some $L \equiv K$ such that
$L$ admits a non-trivial henselian valuation.
\end{Definition}
Equivalently, $t$-henselianity can be axiomatized in 
$\mathcal{L}_\textrm{ring}$ 
via the axiom scheme of
admitting a $t$-henselian topology, see \cite[Remark 7.11]{PZ} and 
\cite[p.~203]{P3}.
In \cite[p.~338]{PZ}, Prestel and Ziegler construct a $t$-henselian field $k$ of characteristic $0$ which is 
not henselian and neither algebraically nor real closed.
In particular, no henselian field $L \equiv k$
can admit any non-trivial $\emptyset$-definable henselian valuation. 
Furthermore, they show that any sufficiently saturated elementary 
extension of a $t$-henselian field is 
henselian (\cite[Theorem 7.2]{PZ}).

\begin{Example}
Let $k$ be a $t$-henselian field of characteristic $0$ 
which is not henselian and neither algebraically nor real closed. 
\begin{itemize}
\item Then
$v_k$ is $\emptyset$-definable as it is trivial. 
However, if $k\prec L$ is an elementary extension with $L$ henselian, then $v_L$ is not $\emptyset$-definable: 
Else, $k$ would also admit a non-trivial $\emptyset$-definable 
henselian valuation, contradicting the assumption that $k$ is not
henselian.
\item
The trivial valuation is not the only example for a
canonical henselian valuation which is $\emptyset$-definable but not $\emptyset$-definable as such: 
Consider $K=k((\mathbb{Z}))$. 
Then $v_K$ is the power series valuation with value group $\mathbb{Z}$.
By a result of Ax (\cite{Ax}), there is an $\mathcal{L}_\textrm{ring}$-formula $\phi(x)$ that uniformly defines
all henselian valuations with value group $\mathbb{Z}$ and residue field of characteristic zero.
Now take once more $L \succ k$ henselian and consider $M=L((\mathbb{Z}))$.
Then, since $L$ is henselian, $v_M$ is a proper refinement of the
power series valuation $w$ on $M$ with residue field $L$ and value group
$\mathbb{Z}$. However, we get $\phi(M)=\mathcal{O}_w$. 
Let now $\psi(x)$ be any other $\mathcal{L}_\textrm{ring}$-formula defining
$\mathcal{O}_{v_K}$ in $K$. Then
$$K \models \forall x (\psi(x) \longleftrightarrow \phi(x)),$$
so any such formula will again define $\mathcal{O}_w$ in $M$.
Hence, $v_K$ is $\emptyset$-definable but not 
$\emptyset$-definable as such.
\end{itemize}
\end{Example}

\begin{Observation} \label{Obs:t}
Let $K$ be a henselian field with $\mathrm{char}(Kv_K)=0$ and assume that
 $v_K$ is $\emptyset$-definable. Then $v_K$ is $\emptyset$-definable as such
iff $Kv_K$ is separably closed or
not $t$-henselian. 
\end{Observation}

\begin{proof} 
Assume that $K$ is a field with $\mathrm{char}(Kv_K)=0$ such that $v_K$ is $\emptyset$-definable, 
say via the $\mathcal{L}_\textrm{ring}$-formula $\phi(x)$. 

Assume first that $v_K$ is 
not $\emptyset$-definable as such. Then,
using Corollary \ref{MC}, we get $v_K \in H_1(K)$.
Furthermore, there is some $L \equiv K$ such that $\phi(L) =: \mathcal{O}_w 
\neq \mathcal{O}_{v_L}$. Since we have $Lw \equiv Kv_K$ and
$v_K \in H_1(K)$, 
$Lw$ is not separably closed,
so we get $\mathcal{O}_w \supsetneq \mathcal{O}_{v_L}$.
In particular, $v_L$ induces
a non-trivial henselian valuation on $Lw$, so $Lw$ is henselian. Hence
$Kv_K$ is $t$-henselian. 

Assume now that $Kv_K$ is not separably closed but $t$-henselian.
Take some $L \succ Kv_K$ henselian and let $u$ be the power series valuation on $K':=L((v_KK))$.
By Ax-Kochen/Ershov (\cite[4.6.4]{P2}), $(K,v_K)\equiv(K',u)$, so $\phi$ defines $\mathcal{O}_u$ in $K'$.
Since $u$ is a proper coarsening of $v_{K'}$, we get that $v_K$ is not $\emptyset$-definable as such.
\end{proof}

Recall that we have shown in Observation \ref{Obshen} that in case we have $v_K
\in H_1(K)$ and $v_K$ is 
$\emptyset$-definable as such, then $v_K$ is
$\emptyset$-$\exists\forall$-definable. 
We use the above Observation to show
that, in general, this definition cannot be improved
when it comes to quantifiers:

\begin{Example} \label{Ex:H1}
We construct a field $K$ with $v_K \in H_1(K)$ such that $v_K$ is
$\emptyset$-definable as such but not $\emptyset$-$\forall\exists$-definable.

Recall that in Example \ref{Ex:H1p} we construct a field 
$k$ with $k \cong k((\mathbb{Q}))(X)$
and extensions
$L_1 := k((u^\mathbb{Q}))((v^\mathbb{Q}))$ and 
$L_2:=k((u^\mathbb{Q}))(X)((v^\mathbb{Q}))$ with valuations $v_1$, $v_2$ such that
$(L_1,v_1) \equiv (L_2,v_2)$ and $L_1\prec_\exists L_2$.
Since $k$ is by construction hilbertian (\cite[13.2.1]{FJ}), 
it is not henselian (nor $t$-henselian,
see \cite[15.5.4]{FJ}) and so the
canonical henselian valuation $v_{L_1}$ on $L_1$ is the power series valuation $v_1$
with residue field $k$ and value group $\mathbb{Q}\oplus \mathbb{Q}$.
Furthermore, $v_{L_1}$ is $\emptyset$-definable (see Example \ref{Ex:H1p}) 
and thus 
$\emptyset$-definable as such 
by Observation \ref{Obs:t}.
On the other hand, the canonical henselian valuation $v_{L_2}$ on $L_2$ is the
power series valuation $v_2$ with residue field $k((u^\mathbb{Q}))(X)$ and value group
$\mathbb{Q}$. Thus, the restriction of $v_{L_2}$ to $L_1$ is a proper coarsening
of $v_{L_1}$ and so $v_{L_1}$ is not $\emptyset$-definable 
by an $\forall\exists$-formula
by Theorem \ref{P}.
\end{Example}

Furthermore, $v_K$ is always $\emptyset$-definable if its residue field
is not $t$-henselian:

\begin{Proposition} \label{Prop:H1}
Let $K$ be a field with $v_K\in H_1(K)$ and $Kv_K$ not $t$-henselian.
Then $v_K$ is $\emptyset$-definable as such by an $\exists\forall$-formula.
\end{Proposition}

\begin{proof} 
Consider the elementary class of valued fields
$$\mathcal{K} := 
\{ (L, v) \;|\; (L, v) \equiv (K,{v_K}) \}.$$
Take $(L_1, v_1)$ and $(L_2, v_2)$ in $\mathcal{K}$ with
$L_1 \prec_\exists L_2$. 
Then $v_1$ is a henselian valuation
on $L_1$ with non-henselian residue field, so $v_1=v_{L_1}$. 
As $Kv_K$ is not separably closed, neither is $Lv_{1} = Lv_{L_1}$ 
and we get $v_{L_1} \in H_1(L_1)$.
Lemma 
\ref{Lemhen} implies that 
the restriction of $v_2$ to $L_1$ is henselian and is hence a
coarsening of $v_{L_1}=v_1$. By Theorem \ref{P}, there is a 
parameter-free $\exists\forall$-formula
defining $\mathcal{O}_v$ in $L$ for any $(L, v) \in \mathcal{K}$.
\end{proof}

We now want to study some assumptions under which $\emptyset$-definability
of $v_K$ and $Kv_K$ $t$-henselian imply that $v_K$ is definable by an 
$\emptyset$-$\forall\exists$-formula.

\begin{Definition} Let $K$ be a field.
A valuation $v$ on $K$ is called \emph{tamely branching at $p$} 
if the value group is
not divisible by $p$, $\mathrm{char}(Kv) \neq p$ and if
$[vK:pvK]=p$, then ${Kv}$ has a finite separable extension of degree divisible 
by $p^2$.
\end{Definition}

\begin{Theorem}[Koenigsmann, {\cite{EP}[5.4.3]}]
A field $K$ admits a henselian valuation, tamely branching at some prime $p$ iff
$G_K$ has a non-procyclic $p$-Sylow subgroup $P \not\cong \mathbb{Z}_2 \rtimes
\mathbb{Z}/2\mathbb{Z}$ with a non-trivial abelian normal closed subgroup
$N$ of $P$. \label{tame}
\end{Theorem}

\begin{Proposition} \label{kQ}
Let $k$ be $t$-henselian 
with 
$v_k$ trivial and
$\mathrm{char}(k)=0$.
Assume that the absolute Galois group $G_k$ of $k$ is small.
Then, both of the following hold: 
\begin{enumerate}
\item For any $L \equiv k$, every henselian valuation $w$ on $L$ with 
$\mathrm{char}(Lw)=0$ has divisible value group. 
In particular, if $L \succ k$, then $v_LL$ is divisible.
\item We have $k\equiv k((\mathbb{Q}))$. 
\end{enumerate}
\end{Proposition}

\begin{proof} 
If $k$ is algebraically closed, then both (1) and (2) are clear. 
Otherwise, the assumption that $v_k$ is trivial implies that
$k$ is not henselian, which we assume now.
\begin{enumerate}
\item Take any $L \equiv k$ and let $w$ be a henselian valuation on $L$ with
$\mathrm{char}(Lw) = 0$. If $L$ is not henselian, then $w$ is trivial and $wL$
is divisible. Hence, we may assume that $L$ is henselian. 
Note that $L$ cannot admit a non-trivial $\emptyset$-definable henselian valuation since
otherwise $k$ would be henselian.
Thus, by \cite[Theorem 3.15]{JK0}, either $L$ is real closed or
every finite group occurs as a subquotient of $G_L$.
In case $L$ is real closed, $wL$ is divisible for any henselian
valuation on $L$ by \cite[4.3.6 and 4.3.7]{EP}. 
In case any finite group occurs as a
subquotient of $G_L$,
the same holds for $G_{Lw}$ (see 
\cite[Observation 3.16]{JK0}).
In particular, $Lw$ has a Galois extension of degree divisible by $p^2$
for every prime $p$.
Assume for a contradiction
that $wL$ is not $p$-divisible. Then $w$ is tamely branching at $p$,
so there is some $p$-Sylow subgroup $P$ of $G_L$ as in Theorem \ref{tame}.
As $G_k$ is small by assumption, 
we get $G_L \cong G_k$ by \cite[20.4.6]{FJ}, so, using Theorem \ref{tame} once more, 
$k$ also admits a non-trivial 
henselian valuation. This contradicts the assumption that $k$ is not henselian. 
Hence, $wL$ is divisible.

The last part now follows since for any $L \succ k$
the restriction of $v_L$ to $k$ is trivial, so we get $\mathrm{char}(Lv_L)=0$.
\item We now 
use Ax-Kochen/Ersov (\cite[4.6.4]{P2}) repeatedly. Note that the lexicographically ordered
direct
sum of two non-trivial
 divisible ordered abelian groups is divisible and hence
elementarily equivalent to $\mathbb{Q}$. 
Take again some $L \succ k$ henselian. 
As the value group of $v_L$ is divisible by the first part, 
we have
$$
 k \equiv L \equiv Lv_L((v_LL)) \equiv Lv_L((v_LL\oplus\mathbb{Q}))\cong Lv_L((v_LL))((\mathbb{Q})) \equiv L((\mathbb{Q})) \equiv k((\mathbb{Q})).
$$
\end{enumerate}
\end{proof}

The following Lemma is probably well-known:
\begin{Lemma} \label{Gamma}
Let $\Gamma$ be an ordered abelian group. The following are equivalent:
\begin{enumerate}
\item $\Gamma$ has a non-trivial $p$-divisible convex subgroup for every prime $p$.
\item $\Gamma$ is elementarily equivalent to an ordered abelian group $\Gamma'$ which has 
a non-trivial divisible convex subgroup.
\item $\Gamma \equiv \mathbb{Q} \oplus \Gamma$.
\item $\Gamma \equiv \mathbb{Q} \oplus \Delta$ for some ordered abelian
group $\Delta$.
\end{enumerate}
\end{Lemma}

\begin{proof}
For a prime $p$, we consider the formula
$$
 \gamma_p(x) \;\equiv\; (x>0) \wedge 
\forall y\; (-x \leq y \leq x \longrightarrow \exists z \;pz=y).
$$
Then, in an ordered abelian group, the sentence
$$
 \phi_p \;\equiv\; \exists x\; \gamma_p(x) 
$$
axiomatizes the existence of a non-trivial 
$p$-divisible convex subgroup. 

(1) $\Rightarrow$ (2): Assume $\Gamma$ has a non-trivial $p$-divisible
convex subgroup for every prime $p$. Since the
convex subgroups of $\Gamma$ are ordered by inclusion, the type
$$
 q(x) = \{  \gamma_p(x) \;| \; p \textrm{ prime}\}
$$
is finitely satisfiable in $\Gamma$. Hence, it is realized in some 
sufficiently saturated $\Gamma' \succ \Gamma$. Now, $\Gamma'$ has a non-trivial
divisible convex subgroup.

(2) $\Rightarrow$ (3): Assume that $\Gamma \equiv \Gamma'$ and that
$\Gamma'$ has a non-trivial 
divisible convex subgroup $D$ with $\Gamma'/D=\Delta$. 
By \cite[bottom of p.~282]{Giraudet}, $\Gamma'\equiv D\oplus\Delta$.
Since $D$ is divisible, $D\equiv\mathbb{Q}\oplus D$.
Thus, since lexicographic products preserve elementary equivalence, cf.~\cite[proof of 3.3]{Giraudet}, we get that
$$
\Gamma \equiv \Gamma'\equiv 
D \oplus \Delta \equiv 
\mathbb{Q} \oplus D \oplus \Delta 
\equiv \mathbb{Q} \oplus \Gamma.
$$

(3) $\Rightarrow$ (4): Immediate. 

(4) $\Rightarrow$ (1): This is clear, since $\mathbb{Q}\oplus\Delta\models\phi_p$ for all primes $p$.
\end{proof}

We can now prove the following:
\begin{Theorem} \label{small}
Assume that $\mathrm{char}(Kv_K)=0$. 
If $v_K$ is $\emptyset$-definable, $Kv_K$ is $t$-henselian
and $G_{Kv_K}$ is small,
then $v_K$ is definable by an $\emptyset$-$\forall\exists$-formula.
\end{Theorem}

\begin{proof} 
Let $\phi(x)$ be the $\mathcal{L}_\textrm{ring}$-formula defining $v_K$.
Note that $Kv_K \equiv Kv_K((\mathbb{Q}))$ holds by Proposition \ref{kQ}.
This implies that
\begin{align} 
 v_KK \not \equiv \mathbb{Q} \oplus v_KK, \label{oplus}
\end{align}
since otherwise, Ax-Kochen/Ersov (\cite[4.6.4]{P2}) gives that
$$
 (K,v_K) \equiv 
 (Kv_K((\mathbb{Q}))\underbrace{((v_KK))}_{u_1},u_1) 
\equiv (Kv_K\underbrace{((\mathbb{Q}))((v_KK))}_{u_2},u_2)
$$
contradicting that $\phi$ defines $v_K$.

\emph{Claim:} If $(L,v)\equiv(K,v_K)$, then $v$ is the coarsest henselian 
valuation on $L$ with $Lv \equiv Kv_K$. 

\emph{Proof of claim:} Assume that $w \supseteq v$ with $Lw \equiv Kv_K$. 
Then
$v$ induces a henselian valuation $\bar{v}$ on $Lw$ with residue field 
$(Lw)\bar{v}=Lv$. In particular, we have $\mathrm{char}((Lw)\bar{v})=0$. 
By Proposition \ref{kQ}, the value group $\bar{v}(Lw)$ of the induced
valuation, which is a convex subgroup of $vL$, is divisible. 
Since $vL\equiv v_KK$,
Lemma \ref{Gamma} together with (\ref{oplus}) above 
now imply that $\bar{v}(Lw)$ is trivial.
Thus we conclude $w=v$.

Take $(L_1,v_1)$, $(L_2,v_2) \equiv (K,v_K)$ with $L_1 \prec_\exists L_2$.
Let $w$ be the restriction of $v_2$ to $L_1$.
By Lemma \ref{Lemhen}, $w$ is henselian.
Note that $w$ and $v_1$ are comparable:
If $v_1 \in H_1(L_1)$ then $v_1$ is comparable to any henselian valuation on 
$L_1$ (cf.~section 2). In case we have $v_1 \in H_2(L_1)$, we get
-- using the Claim -- that $v_1$ is the coarsest henselian valuation on $L_1$
with algebraically closed residue field. Thus, we have $v_1=v_{L_1}$ and
so again $v_1$ is comparable to any henselian valuation on $L_1$.

Now, assume for a contradiction 
that $w$ is a 
proper 
coarsening of $v_1$.
Then $v_1$ induces a henselian valuation $\bar{v}_1$ on $L_1w$
with 
residue field $(L_1w)\bar{v}_1=L_1v_1\equiv Kv_K$ and
value group $\Delta:=\bar{v}_1(L_1w)$ a nontrivial convex subgroup of $v_1L_1$.
By Ax-Kochen/Ersov, $L_1w \equiv Kv_K((\Delta))$,
and the Claim gives that $L_1w \not\equiv Kv_K$.
By (\ref{oplus}) and Lemma \ref{Gamma}, $\Delta$ is not divisible.
Recall 
that $L_1w$ is relatively algebraically closed in 
$L_2v_2$ by Lemma \ref{Lemhen}. 
Thus, the restriction homomorphism
$$
 r:G_{L_2v_2} \longrightarrow G_{L_1w} 
$$
is surjective. By \cite[5.2.6]{EP}, the residue homomorphism induced by $\bar{v}_1$,
$$
 \pi:G_{L_1w} \longrightarrow G_{L_1v_1},
$$
is also surjective.
Since $G_{Kv_K}$ is small by assumption, we have
$G_{Kv_K}\cong G_{L_1v_1}\cong G_{L_2v_2}$ 
(see \cite[20.4.6]{FJ}),
so the epimorphism
$$
 \pi\circ r:G_{L_2v_2} \longrightarrow G_{L_1v_1}
$$
is actually an isomorphism (\cite[16.10.6]{FJ}), implying that both $r$ and $\pi$ are isomorphisms.
In particular,
$$
 G_{Kv_K} \cong G_{L_1w} \cong G_{Kv_K((\Delta))}.
$$
If $I$ denotes the inertia group of the power series valuation on 
$Kv_K((\Delta))$,
then reduction gives an homomorphism $\pi:G_{Kv_K((\Delta))}\rightarrow G_{Kv_K}$ with kernel $I$ (cf.~\cite[5.2.6]{EP}).
Since $G_{Kv_K}$ is small and $G_{Kv_K((\Delta))}\cong G_{Kv_K}$ this implies that $I=1$.
As $I=\prod_p\mathbb{Z}_p^{d_p}$ with $d_p={\rm dim}_{\mathbb{F}_p}(\Delta/p\Delta)$ (see \cite[5.3.3]{EP}), we conclude that $\Delta$ is divisible,
a contradiction.
Therefore, $w$ is a refinement of $v_1$,
so the claim follows from 
Theorem \ref{P}.
\end{proof}

Note that we construct fields which are $t$-henselian but not henselian with
small absolute Galois group in the last section. Thus, the hypotheses of the 
above Theorem is not empty. Overall, we can now combine several of our results
to get the following Theorem as stated in the introduction:
{
\renewcommand{\theTheorem}{\ref{main}}
\begin{Theorem}
Let $K$ be a field with canonical henselian valuation $v_K$ whose residue field $F=Kv_K$ has characteristic zero.
Assume that $v_K$ is $\emptyset$-definable.
\begin{enumerate}
\item If $F$ \underline{is not} elementarily equivalent to a henselian field, then $v_K$ is $\emptyset$-$\exists\forall$-definable.
\item If $F$ \underline{is} elementarily equivalent to a henselian field,
 then $v_K$ is $\emptyset$-$\forall\exists$-definable if the absolute Galois group $G_F$ of $F$ is a small profinite group.
\end{enumerate}
\end{Theorem}
\begin{proof}
Let $K$ be a field with $\mathrm{char}(Kv_K)=0$ and assume that $v_K$ is
$\emptyset$-definable. Then, case (1) 
is a special case of
Proposition \ref{Prop:H1}.
Case (2) follows immediately from Theorem \ref{small}.
\end{proof}
}

\section{$t$-henselian non-henselian fields with small absolute Galois group}

We refine the construction sketched in \cite[p.~338]{PZ} of a $t$-henselian field which is neither henselian nor real closed.

\begin{Definition}
Let $n\in\mathbb{N}$. We say that a valued field $(K,v)$ is $n_\leq$-henselian if every
monic $f\in\mathcal{O}_v[T]$ of degree at most $n$ for which $\bar{f}\in Kv[T]$ has a simple zero $a\in Kv$
has a zero $x\in\mathcal{O}_v$ with $\bar{x}=a$.
\end{Definition}

Note that $(K,v)$ is henselian if and only if it is $n_\leq$-henselian for all $n$.

\begin{Lemma}\label{lem:composenhens}
Let $v_1$ be a valuation on $K$ and $v_2$ a valuation on $Kv_1$.
If both $v_1$ and $v_2$ are $n_\leq$-henselian, then so is the valuation $v=v_2\circ v_1$ on $K$.
\end{Lemma}

\begin{proof}
Let $f\in\mathcal{O}_v[T]$ monic of degree at most $n$ such that $\bar{f}\in Kv[T]$ has a simple zero $a\in Kv$.
First lift $a$ to a zero $a'\in Kv_1$ of the reduction of $f$ with respect to $v_1$, and then further to a zero $x\in\mathcal{O}_v$ of $f$.
\end{proof}

\begin{Lemma} Let $(K,v)$ be a valued field.\label{lem:nhensextend}
 \label{lem:nchar}
\begin{enumerate}
\item
If every polynomial \label{lem:nchar1}
$$
 g=T^m+T^{m-1}+\sum_{i=0}^{m-2}a_iT^i\in K[T]
$$ 
with $m \leq n!$ and $a_0,\dots,a_{m-2}\in\mathfrak{m}_v$ has a 
zero $x\in\mathcal{O}_v$ with $x+1\in\mathfrak{m}_v$, then
$v$ extends uniquely to every Galois extension $N|K$ with $[N:K]\leq n$.
\item
If $v$ extends uniquely to every Galois extension $N|K$ with $[N:K]\leq n!$, then
$(K,v)$ is $n_{\leq}$-henselian.
\end{enumerate}
\end{Lemma}

\begin{proof} 
The proof follows by standard arguments.
Part (1) follows immediately from the proof of $(6)\Rightarrow(1)$ in 
\cite[4.1.3]{EP}.

Assume now that the assumption of (2) holds. 
Let $f\in\mathcal{O}_v[T]$ be monic of degree at most $n$ for which 
$\bar{f}\in Kv[T]$ has a simple zero $a\in Kv$. 
We may assume that $f$ is irreducible over $\mathcal{O}_v$, hence, 
by Gauss' Lemma \cite[4.1.2(1)]{EP}, also over $K$.
Consider the splitting field $L$ of $f$ over $K$.
Then $[L:K]\leq n!$,
so by assumption there is a unique extension $w$ of $v$ to $L$.
There are $a_1,\dotsc,a_n \in L$ with 
$f = \prod_{i=1}^{n} (T-a_i)$.
By Gauss' Lemma, $a_1,\dots,a_n\in\mathcal{O}_w$, 
and without loss of generality we can assume that $\bar{a}_1=a$. 
Suppose for a contradiction that we have $n>1$.
Then there is some $\sigma \in \mathrm{Gal}(L|K)$ with 
$\sigma(a_1)=a_2$.
As $w$ is the unique extension of $v$ to $L$, we have 
$\sigma(\mathcal{O}_w)= \mathcal{O}_w$.
Thus, $\sigma$ induces an automorphism 
$\bar{\sigma} \in
 \mathrm{Gal}(Lw|Kv)$ such that 
$\bar{a}_2=\bar{\sigma}(\bar{a}_1)=\bar{\sigma}(a)=a$
holds. This
contradicts the fact that $a$ is a simple zero of $\bar{f}$.
\end{proof}

We denote by $\mathbb{P}$ the set of prime numbers.

\begin{Lemma}\label{prop:construction}
Let $K_0$ be a field of characteristic zero that contains all roots of unity.
Let $n\in\mathbb{N}$, $n<q\in\mathbb{P}$ and $P\subseteq\mathbb{P}$.
Then there exists a valued field $(K_1,v)$ with the following properties:
\begin{enumerate}
\item $K_1v=K_0$ and $vK_1=\mathbb{Z}[\frac{1}{p}:p\in\mathbb{P}\setminus P]$
\item $v$ is $n_\leq$-henselian but not $q$-henselian.
\item $G_{K_1}=\left<H_1,H_2\right>$, where $H_1\cong\mathbb{Z}_q$ and there is $N\lhd H_2$ closed with 
 $N\cong\prod_{p\in P}\mathbb{Z}_p$ and $H_2/N\cong G_{K_0}$.
\end{enumerate}
\end{Lemma}

\begin{proof}
Let $\Gamma=\mathbb{Z}[\frac{1}{p}:p\in\mathbb{P}\setminus P]$, $F_0=K_0(x)$, $F=K_0(x^\Gamma)\subseteq K_0((x^\Gamma))$,
and $F^h=F^{\rm alg}\cap K_0((x^\Gamma))$.
On all subfields of $K_0((x^\Gamma))$ we denote the restriction of the $x$-adic power series valuation by $v$.
Then $K_1v=K_0$ for all $K_0\subseteq K_1\subseteq K_0((x^\Gamma))$ and $vK_1=\Gamma$ for all $F\subseteq K_1\subseteq K_0((x^\Gamma))$.

Let 
$$
 f(T)=T^q-(x+1)\in F_0[T].
$$ 
Since $\bar{f}=T^q-1$ is separable and completely decomposes over $F_0v=K_0$, 
$f$ has a zero $\alpha\in F^h\subseteq K_0((x^\Gamma))$ by Hensel's Lemma.
Since $f$ is irreducible over $F_0$, $F_0(\alpha)|F_0$ is a $C_q$-extension.
The fact that $v(x+1)=0$ implies that $v$ does not ramify in this extension 
\cite[2.3.8]{FJ}.
Since $v$ is totally ramified in $F|F_0$ but unramified in $F_0(\alpha)|F_0$,  these extensions are linearly disjoint over $F_0$, cf.\ \cite[2.5.8]{FJ}, 
hence also $F(\alpha)|F$ is a $C_q$-extension, cf.\ \cite[2.5.2]{FJ}.

Now let ${\rm res}:G_F\rightarrow{\rm Gal}(F(\alpha)|F)$ be the restriction homomorphism and let $Q\leq G_F$ be a $q$-Sylow subgroup.
Then ${\rm res}(Q)$ is a $q$-Sylow subgroup of ${\rm Gal}(F(\alpha)|F)\cong C_q$, so there exists $\sigma\in Q$ with $\left<{\rm res}(\sigma)\right>={\rm Gal}(F(\alpha)|F)$.
The procyclic group $G_q:=\left<\sigma\right>$ is torsion-free since it is the absolute Galois group of a non-real field,
and pro-$q$ as a subgroup of $Q$,
hence $G_q\cong\mathbb{Z}_q$, cf.\ \cite[Ch. 1 Exercise 7]{FJ}.

Let $E$ denote the fixed field of $G_q$ and $K_1=E\cap F^h$.
Then $G_{K_1}=\left<G_q,G_{F^h}\right>$.
By \cite[5.3.3]{EP}, the absolute inertia group $I_v$ of the valuation on $F^h$ satisfies
$$
 I_v \cong \prod_{p\in\mathbb{P}}\mathbb{Z}_p^{\dim_{\mathbb{F}_p}(\Gamma/p\Gamma)}=\prod_{p\in P}\mathbb{Z}_p,
$$
and $G_{F^h}/I_v\cong G_{F^hv}=G_{K_0}$.

Since $E\cap F(\alpha)=F$, $K_1(\alpha)$ is a $C_q$-extension of $K_1$ contained in $F^h$, so $(K_1,v)$ is not $q$-henselian, cf.\ \cite[Proposition 1.2(iv)]{Koe}.
If $g\in(\mathcal{O}_v\cap K_1)[T]$ is monic of degree at most $n$ and $\bar{g}\in K_1v[T]$ has a simple zero $a\in K_1v$,
then $g$ has a zero $\beta$ in $F^h$ with $\bar{\beta}=a$ by Hensel's Lemma.
Since $[E(\beta):E]\leq{\rm deg}(g)\leq n<q$ and $G_E\cong\mathbb{Z}_q$, we conclude that $\beta\in E\cap F^h=K_1$,
so $(K_1,v)$ is indeed $n_\leq$-henselian.
\end{proof}

\begin{Construction}\label{Con}
Fix a prime number $p_0$ and let $K_0=\mathbb{C}$. 
For $n=1,2,\dots$ choose a prime number $q_n>\max\{n,p_0\}$ and iteratively use Lemma \ref{prop:construction} (with $P=\emptyset$) to construct a valued field $(K_n,v_n)$ with
$v_nK_n=\mathbb{Q}$, $K_nv_n=K_{n-1}$, $K_n$ $n_\leq$-henselian but not $q_n$-henselian,
and $G_{K_n}=\left<G_{K_n}',G_{K_n}''\right>$ with $G_{K_n}'\cong\mathbb{Z}_{q_n}$, $G_{K_n}''\cong G_{K_{n-1}}$.
By induction, $G_{K_n}$ is finitely generated, in particular small.

For each $n\geq m$, composition of places gives a valuation $v_{n,m}=v_{m+1}\circ\dots\circ v_n$ on $K_n$ with residue field $K_nv_{n,m}=K_m$.
Since $v_nK_n$ is divisible and the class of divisible ordered abelian groups is closed under extensions,
induction shows that $v_{n,0}K_n$ is divisible
for all $n$.

The residue homomorphism $\mathcal{O}_{v_{n,m}}\rightarrow K_m$ of $v_{n,m}$ restricts to a homomorphism $\mathcal{O}_{v_{n,0}}\rightarrow\mathcal{O}_{v_{m,0}}$.
With respect to these homomorphisms, the $\mathcal{O}_{v_{n,0}}$ form an inverse system.
The inverse limit $\mathcal{O}=\varprojlim_n\mathcal{O}_{v_{n,0}}$ is again a valuation ring, cf.\ \cite[Lemma 3.5]{FehmParan}.
Let $K={\rm Quot}(\mathcal{O})$ and $v$ a valuation such that $\mathcal{O}=\mathcal{O}_v$.
For each $n$, let $\mathfrak{p}_n$ denote the kernel of the natural projection $\mathcal{O}\rightarrow\mathcal{O}_{v_{n,0}}$
and let $v_n^*$ be a valuation on $K$ with $\mathcal{O}_{v_n^*}=\mathcal{O}_{\mathfrak{p}_n}$.
Note that $\bigcap_n\mathfrak{p}_n=(0)$, hence $\bigcup_n\mathcal{O}_{v_n^*}=\mathcal{O}_{\bigcap_n\mathfrak{p}_n}=\mathcal{O}_{(0)}=K$
and 
$$
 vK=K^\times/\mathcal{O}_v^\times=\bigcup_n\mathcal{O}_{v_n^*}^\times/\mathcal{O}_v^\times.
$$ 
As $\mathcal{O}_{v_n^*}^\times/\mathcal{O}_v^\times\cong K_n^\times/\mathcal{O}_{v_{n,0}}^\times=v_{n,0}K_n$,
we see that $vK$ is divisible.
\end{Construction}

\begin{Lemma}\label{lem:finitequotients}
Let $G$ be a profinite group generated by closed subgroups $G_0$, $G_1$.
If $G_0$ is pro-$q$ and $A$ is a finite group with $q\nmid\#A$ which is a quotient of $G$,
then $A$ is also a quotient of $G_1$.
\end{Lemma}

\begin{proof}
Let $\pi:G\rightarrow A$ be an epimorphism. Then $\pi(G_0)$ is a $q$-group, so $G_0\subseteq{\rm ker}(\pi)$ since $q\nmid\#A$.
In particular, $G=\left<G_1,{\rm ker}(\pi)\right>$,
so the inclusion $G_1\rightarrow G$ induces an epimorphism $G_1\rightarrow G/{\rm ker}(\pi)\cong A$.
\end{proof}

\begin{Proposition} \label{Prop:small}
The field $K$ of Construction \ref{Con} is $t$-henselian but not henselian,
$G_K$ is small and $K(p_0)=K$.
\end{Proposition}

\begin{proof}
For $n\in\mathbb{N}$ let $l_n := \mathrm{max}\{ l \in \mathbb{N}\,:\, l! \leq n\}$. 
Observe that each $v_n^*$ is $(l_n)_\leq$-henselian: 
By Lemma \ref{lem:nchar} it suffices to show that each
$$
 g=T^m+T^{m-1}+\sum_{i=0}^{m-2}a_iT^i\in K[T]
$$ 
with $m\leq n$ and $a_i\in\mathfrak{m}_{v_n^*}$ for $i=0,\dots,m-2$ has a zero $x$ in $\mathcal{O}_{v_n^*}$ with $x+1\in\mathfrak{m}_{v_n^*}$.
For $k\in\mathbb{N}$ let $g_k$ denote the reduction of $g$ with respect to $v_k^*$.
If $k\geq n$, then the reduction of $g_k$ with respect to $v_{k,n}$ is $g_n=T^m+T^{m-1}$,
so since $v_{k,n}$ is $n_\leq$-henselian by Lemma \ref{lem:composenhens}, 
the simple zero $x_n=-1$ uniquely lifts to a zero $x_k\in\mathcal{O}_{v_{k,n}}$ of $g_k$.
Since
$x_n\in\mathcal{O}_{v_{n,0}}$, also $x_k\in\mathcal{O}_{v_{k,0}}$. 
Therefore, $x=(x_k)_k\in\mathcal{O}$ satisfies
$g(x)\in\bigcap_k\mathfrak{p}_k=(0)$ and $x+1\in\mathfrak{m}_{v_n^*}$.
This concludes the proof that $v_n^*$ is $(l_n)_\leq$-henselian.

As $l_n\rightarrow\infty$ for $n\rightarrow\infty$, 
\cite[Theorem 7.2]{PZ} implies that
the topology induced by each of the $v_n^*$ (for $n>1$) 
on $K$ is $t$-henselian.

However, $K$ is not henselian:
Suppose that $w$ is a non-trivial henselian valuation on $K$.
Since the topology induced on $K$ by $w$ coincides with the $t$-henselian topology induced by each of the $v_n^*$ (\cite[Theorem 7.9]{PZ}),
and the valuation ring $\bigcup_n\mathcal{O}_{v_n^*}=K$ is trivial,
\cite[2.3.5]{EP} implies that there is some $n$ with $\mathfrak{m}_{v_n^*}\subseteq\mathfrak{m}_w$,
i.e.~$\mathcal{O}_w\subseteq\mathcal{O}_{v_n^*}$. In particular, $v_n^*$ is henselian.
This implies that also the valuation induced by $v_n^*$ on $Kv_{n+1}^*=K_{n+1}$ is henselian,
but this valuation is exactly $v_{n+1}$, which is not $q_{n+1}$-henselian by construction.

We claim that $G_K$ is small: Indeed, otherwise there exist infinitely many distinct extensions $L_1,L_2,\dots$ of $K$ of the same degree $d$.
Without loss of generality we may assume that all $L_i|K$ are Galois.
Fix $k\in\mathbb{N}$ and let $M_k=L_1\cdots L_k$ be the compositum.
Then $A_k:={\rm Gal}(M_k|K)$ is a subgroup of $\prod_{i=1}^k{\rm Gal}(L_i|K)$, so $\#A_k|d^k$.
Choose $n$ with $l_n\geq \max\{|A_k|,d\}$. Since $v_n^*$ is 
$(l_n)_\leq$-henselian, it extends uniquely to $M_k$ by Lemma~\ref{lem:nhensextend}.
Since $v_n^*K$ is divisible, this extension is unramified, hence the fundamental equality \cite[3.3.3]{EP} gives that
${\rm Gal}(M_k|K)\cong{\rm Gal}(M_kv_n^*|K_n)$. 
In particular, $A_k$ is a quotient of $G_{K_n}$.
For all $m=d,\dots,n$ we have that $q_m>m\geq d$,
hence $q_m\nmid\#A_k$, so
Lemma \ref{lem:finitequotients} shows that
$A_k$ is a quotient also of $G_{K_d}$.
Since $k$ was arbitrary and $A_k$ has at least $k$ distinct quotients of order $d$,
this contradicts that $G_{K_{d}}$ is small.

Similarly, $K(p_0)=K$: Indeed, otherwise let $M|K$ be a $C_{p_0}$-extension.
Since $q_m>p_0$ for all $m$, the argument of the previous paragraph shows that there is a 
$C_{p_0}$-extension $M_0$ of $K_0$, contradicting our choice of $K_0$.
\end{proof}

\begin{Example} \label{Ex:t} 
We construct a field $K$ with $v_K \in H_1(K)$, 
$\mathrm{char}(Kv_K)=0$ and
such that $v_K$ is $\emptyset$-$\forall\exists$-definable but
not $\emptyset$-$\exists\forall$-definable. 
Note that Observation \ref{Obshen} implies that
in this case $v_K$ cannot be
$\emptyset$-definable as such. Furthermore, by Observation \ref{Obs:t}, 
for any such field $K$, we have $Kv_K$ $t$-henselian. 

By Construction \ref{Con} and Proposition \ref{Prop:small}, 
for any prime $p$ there is a field $k$ which 
\begin{itemize}
\item is $t$-henselian but not henselian,
\item has characteristic $0$, 
\item satisfies $k=k(p)$ and
\item has small absolute Galois group.
\end{itemize}
We now repeat the construction from Example \ref{Ex:H2}. 
Define again $H_i=\mathbb{Z}\oplus\mathbb{Q}$,
$\Gamma=H_1\oplus H_2\oplus\dots$,
$G_1:=\mathbb{Q}\oplus \Gamma$,
$G_2:=H_0\oplus\Gamma$,
$K_1:=k((G_1))$
and
$K_2:=k((G_2))$.
As in Example \ref{Ex:H2}, we have
 $G_1\prec_\exists G_2$ and $K_1\prec_\exists K_2$.
Let $v_i$ denote the valuation on 
$K_i$ with value group $G_i$ residue field $k$,
and let $w$ denote the valuation on $K_1$ with value group $\Gamma$ and residue field $k((\mathbb{Q}))$.
By Proposition \ref{kQ}, we have $k\equiv k((\mathbb{Q}))$.
Therefore, the Ax-Kochen/Ersov Theorem (\cite[4.6.4]{P2}) implies
$$(K_1,w)\equiv(K_2,v_2).$$

We now have $v_{2} = v_{K_2} =v_{K_2}^p$
as $k=k(p)$ holds and as $\Gamma$ has no $p$-divisible convex subgroup.
Thus, $v_2$ is $\emptyset$-$\forall\exists$-definable by Proposition \ref{p2}.
Just like in Example
\ref{Ex:H2}, the restriction of $v_2$ to $K_1$ gives $v_1$
which is a proper refinement of
$w$. 

Thus, $v_2= v_{K_2}$ is $\emptyset$-$\forall\exists$-definable but
not 
$\emptyset$-$\exists\forall$-definable (see Theorem \ref{P}).
\end{Example}

\section*{Acknowledgements}
\noindent 
The authors would like to thank Jochen Koenigsmann and Alexander Prestel 
for many interesting discussions on the subject.
In particular, an earlier (unpublished) proof of 
Jochen Koenigsmann gave us the idea to 
prove Proposition 5.5.
We would also like to thank Immanuel Halupczok 
for providing an example of a specific ordered abelian group
as well as Volker Weispfenning for presenting his results in \cite{We} in an exceptionally clear and helpful way.

\end{document}